\documentclass[11pt, oneside]{amsart}   	
\usepackage{geometry}                		
\usepackage{amssymb}
\usepackage{hyperref}
\usepackage{multirow}

\usepackage{todonotes}
\usepackage{amssymb,amsthm,amsmath}
\usepackage{hyperref}
\usepackage{xypic}
\usepackage{tikz}
\usetikzlibrary{calc}
\usepackage{rotating}
\usepackage{amssymb}
\usepackage{xypic}
\usepackage{tikz}
\usetikzlibrary{calc}

\usepackage[all, cmtip]{xy}
\usepackage{mathrsfs}  

\usepackage{eucal,accents,upgreek,enumerate}
\usepackage[headings]{fullpage}
\usepackage{bm}

\newtheorem{theorem}{Theorem}
\numberwithin{theorem}{section}
\newtheorem{proposition}[theorem]{Proposition}
\newtheorem{example}[theorem]{Example}
\newtheorem{lemma}[theorem]{Lemma}
\newtheorem{corollary}[theorem]{Corollary}
\newtheorem{conjecture}{Conjecture}

\theoremstyle{definition}
\newtheorem{definition}[theorem]{Definition}
\newtheorem{remark}[theorem]{Remark}

\newtheorem*{acknowledgements}{Acknowledgements}

\newtheorem*{definition*}{Definition}
\newtheorem*{theorem*}{Theorem}
\newtheorem*{proposition*}{Proposition}
\newtheorem*{conjecture*}{Conjecture}
\newtheorem*{question*}{Question}

\DeclareMathOperator{\pic}{Pic}
\DeclareMathOperator{\prin}{Prin}
\DeclareMathOperator{\Jac}{Jac}
\DeclareMathOperator{\trop}{Trop}
\DeclareMathOperator{\val}{val}
\DeclareMathOperator{\dv}{Div}
\DeclareMathOperator{\ddiv}{div}

\newcommand \ZZ {\mathbb{Z}}

\newcommand\Mp {\mathbb{M}^{\text{planar}}}

\usetikzlibrary{positioning}

\tikzset{main node/.style={circle,fill=blue!20,draw,minimum size=1.3cm,inner sep=0pt},
}

\begin{document}
	\title{\textsc{Bitangents of non-smooth tropical quartics}}
	\author{Heejong Lee and Yoav Len}
	\maketitle
	
	\begin{abstract}	
		We study bitangents of non-smooth  tropical plane quartics. Our main result is that with appropriate multiplicities, every  such curve has 7 equivalence classes of bitangent lines. Moreover, the multiplicity of bitangent lines varies continuously in families of tropical plane curves.
			\end{abstract}
	
	\section{Introduction}
	The topic of bitangents to  quartics curves has gained renewed attention recently with the introduction of tropical techniques. As shown in \cite{BLMPR}, every smooth tropical plane quartic admits seven families of bitangent lines. The relation between tropical and algebraic bitangents was later studied in \cite{Chan2, JL, LM, Panizzut}. 
	The goal of the current note is to begin an examination of the bitangents when the curve is singular. In this case, the genus of the tropical curve may not be $3$ anymore, and the number of bitangent lines may go down. However, there is a natural way of assigning a multiplicity to each of them (see Definition \ref{def:bitangentMultiplicity}). Our main result is as follows. 
	
	\begin{theorem*}[\ref{thm:mainTheorem}]
		Let $\Gamma$ be a tropical plane quartic of genus $g$. Then $\Gamma$ has $2^g$ equivalence classes of bitangent lines. One of them has multiplicity $2^{3-g}-1$, and each of the others has multiplicity $2^{3-g}$.
	\end{theorem*}
	\noindent In particular, when $g=3$, there are $7$ bitangent lines with multiplicity $1$, and a single bitangent with multiplicity $0$, in accordance  \cite[Theorem 3.9]{BLMPR}.
	\noindent	
	\medskip
	
	Non-hyperelliptic smooth algebraic curves of genus $3$ are canonically embedded in the plane. Therefore, every line section of the curve gives rise to the canonical divisor, and their bitangents are in bijection with half canonical divisors. An analogous argument works for smooth tropical quartics. However, when the tropical curve is no longer smooth, and especially when its genus drops, this bijection falls apart. To overcome this difficulty, we parametrize the tropical quartic using a metric graph of genus 3, which we refer to as the \emph{paired metric graph}, and obtain an upper bound on the number of bitangent lines. 
	
	On the other hand, to obtain a lower bound on the number of bitangents, we use a limiting process. Given a non-smooth quartic, we approximate it using a sequence of smooth quartics. 
	Via upper semicontinuity of intersection numbers, each of the $7$ bitangent lines of the smooth quartic converges in the limit to a bitangent of our non-smooth quartic. 
	
	%
	%
	\begin{acknowledgements}
		This paper arose from the second author's project group at the 2017 Fields Undergraduate Summer Research Program funded by Fields Institute. We thank the Fields Institute for this prestigious opportunity. We wish to thank the other member of the group, Run Yu Tan,
		Akshay  Tiwary, and Wanlong Zheng for insightful comments and fruitful discussions, and we thank  Ralph Morrison for helpful discussions. 
	\end{acknowledgements}
	
	\section{Preliminaries}
	In this section, we provide a brief overview of some of the basic terms in tropical geometry that will be used throughout. See \cite[Chapter 1]{MS} for more detail. 
	
	\subsection{Tropical plane curves}
	
	A tropical plane curve $\Gamma$ is a planar metric graph possibly with unbounded edges, together with a weight function $w_E : E \rightarrow \ZZ_{> 0}$, such that each edge has a rational slope, satisfying the following balancing condition: the weighted sum of the outgoing primitive integral vectors emanating from every vertex is zero. The last condition is called the \emph{balancing condition}. For the rest of the paper, $\Gamma$ will denote a plane tropical curve. 
	
	Every tropical curve is dual to a regular subdivision of a lattice polygon, known as its \emph{subdivided Newton polygon}. Each edge of the curve corresponds to an edge of the Newton polygon, and each vertex corresponds to a face. The weight of an edge is the lattice length of the corresponding edge of the polygon. Note that the Newton polygon determines the directions and weights of the edges of the tropical curve, but not the length. 
	%
	
	\begin{definition}
	A tropical curve $\Gamma$ has degree $d$ if its Newton polygon is a standard right angled triangle triangle of edge length $d$. 
	\end{definition}
	
\noindent	In particular, a tropical quartic is dual to a right angled triangle of edge length $4$. 
Note that a general tropical line meets a tropical curve of degree $d$ at $d$ points. If the Newton polygon consists only of triangles with area $\frac{1}{2}$, then the curve is said to be $\emph{smooth}$. 
	
	\subsection{Singularities of tropical plane curves} 
Each vertex of $\Gamma$ corresponds to a face of the Newton polygon. The curve is said to be \emph{singular} at $v$ if the area of the dual face is strictly greater than $\frac{1}{2}$. The following theorem determines precisely when that occurs \cite{Pick}.
\begin{theorem}[Pick's theorem]
		Let P be a lattice polygon with area $A$. Then
		\[A=\frac{b}{2} + i - 1,\]
		where b denotes the number of lattice points on the boundary and i denotes the number of lattice points in the interior.
	\end{theorem}

\noindent Since every polygon has at least three vertices, its area is greater than $\frac{1}{2}$ if and only if it has additional lattice points either on the boundary or in the interior. 
		 Figure 1 shows three different ways to obtain singularities. In the first case, the edge adjacent to the vertex has weight $2$. In the second case, the vertex is $4$-valent.  The third type of singularity occurs  when the face dual to the vertex contains a lattice point in its interior, which means that the genus of the entire curve has dropped. We keep track of the lost genus by adding integer weights on the vertices. 
	
	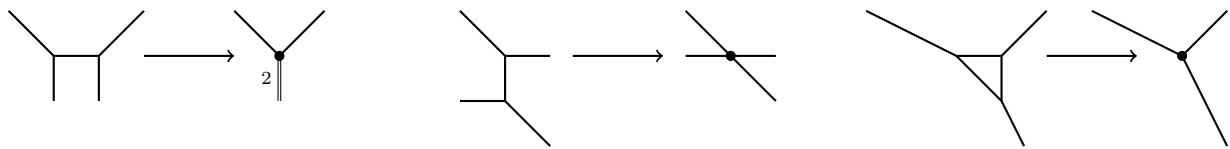
\begin{figure}[h]
		\begin{tikzpicture}[scale=.6,color=black]
		
		\begin{scope}
		
		\draw[thick] (-1,1) -- (0,0);
		\draw[thick] (0,0) -- (1,0);
		\draw[thick] (2,1) -- (1,0);
		\draw[thick] (1,-1) -- (1,0);
		\draw[thick] (0,0) -- (0,-1);
		\draw[thick, ->] (2,0) -- (4,0);
		\end{scope}

		\begin{scope}[shift={(5,0)}]
		
		\draw[thick] (-1,1) -- (0,0);
		\draw[thick] (0,0) -- (1,1);
		\draw[double] (0,0) -- (0,-1);
		\tiny\node [left] at (0,-.5) {2};
		\draw [fill=black,radius=.1] (0,0) circle;
		\end{scope}
		
		\begin{scope}[shift={(10,0)}]
		\draw[thick] (-1,1) -- (0,0);
		\draw[thick] (0,-1) -- (0,0);
		\draw[thick] (1,0) -- (0,0);
		\draw[thick] (0,-1) -- (-1,-1);
		\draw[thick] (1,-2) -- (0,-1);
		\draw[thick, ->] (1.5,0) -- (3.5,0);
		\end{scope}

		\begin{scope}[shift={(15,0)}]
		\draw[thick] (-1,1) -- (1,-1);
		\draw[thick] (-1,0) -- (1,0);
		\draw [fill=black,radius=.1] (0,0) circle;
		\end{scope}
		
		\begin{scope}[shift={(20,0)}]
		\draw[thick] (-2,1) -- (0,0);
		\draw[thick] (0,0) -- (1,-1);
		\draw[thick] (0,0) -- (1,0);
		\draw[thick] (1,-1) -- (1,0);
		\draw[thick] (1,-1) -- (1.5,-2);
		\draw[thick] (2,1) -- (1,0);
		\draw[thick, ->] (2,0) -- (4,0);
		\end{scope}
		
		\begin{scope}[shift={(25,0)}]
		\draw[thick] (-2,1) -- (0,0);
		\draw[thick] (0,0) -- (1,-2);
		\draw[thick] (0,0) -- (1,1);
		\draw [fill=black,radius=.1] (0,0) circle;
		\end{scope}

		\end{tikzpicture}
		\caption{Singular vertices on tropical plane curve.}
		\label{fig:singularity}
	\end{figure}

	\begin{definition}
		 The \emph{weight} function $w_\Gamma:\Gamma\to\mathbb{Z}_{\geq 0}$ assigns to each vertex of $\Gamma$ the number of lattice points in the interior of its dual face, and is zero everywhere else. 
\end{definition}
	
	Given a tropical plane curve $\Gamma$, we parametrize it with a trivalent metric graph $\Sigma$ (possibly with infinite edges). The vertex set of $\Sigma$ is the same as $\Gamma$. Every edge of weight $w_e$ of $\Gamma$ is replaced with $w_e$ edges of $\Sigma$ of the same length. Whenever a vertex $v$ of $\Gamma$ has weight $w_v>0$, we attach $w_v$ loops of length $\epsilon>0$ based at the corresponding vertex of $\Sigma$. None of our results below depend on the length of these loops.
	We have a natural map $\phi:\Sigma\to\Gamma$ which collapses loops and multiple edges. See Figure \ref{fig:pairedmetricgraph} for an example.
	
	\begin{figure}\label{fig:pairing}
		\begin{tikzpicture}[scale=.5 , color=black]
		\begin{scope}
		\draw[thick] (-1,1) -- (0,0);
		\draw[thick] (0,-1) arc (-30:30:1cm);
		\draw[thick] (0,-1) arc (210:150:1cm);
		\draw[thick] (-1,-2) -- (-2,-2);
		\draw[thick] (1,1) -- (0,0);
		\draw[thick] (-1,-2) -- (0,-1);
		\draw[thick] (0,-1) -- (2,-3);
		\draw[thick] (-1,-2) -- (-1,-3);
		\draw[thick] (-1,-3) -- (2,-3);
		\draw[thick] (-1,-3) -- (-2,-4);
		\draw[thick] (2,-3) -- (4,-4) -- (5,-3);
		\draw[thick] (4,-4) -- (4.5,-5);
		\draw[thick] (4,-4) arc (-30:30:1cm);
		\draw[thick] (4,-4) arc (210:150:1cm);
		\draw[thick, ->] (7,-2) -- (10,-2);
		\tiny\node [above] at (8.5,-2) {Pairing morphism $\phi$};
		\tiny\node [below] at (1,-5) {Paired metric graph $\Sigma$};
		
		\tiny\node [left] at (0,0) {$P$};
		\tiny\node [left] at (0,-1) {$Q$};
		\tiny\node [below left] at (4,-4) {$R$};

		\end{scope}
		
		\begin{scope}[shift={(15,0)}]
		\draw[thick] (-1,1) -- (0,0);
		\draw[double] (0,-1) -- (0,0);
		\tiny\node [left] at (0,-.5) {2};
		\draw[thick] (-1,-2) -- (-2,-2);
		\draw[thick] (1,1) -- (0,0);
		\draw[thick] (-1,-2) -- (0,-1);
		\draw[thick] (0,-1) -- (2,-3);
		\draw[thick] (-1,-2) -- (-1,-3);
		\draw[thick] (-1,-3) -- (2,-3);
		\draw[thick] (-1,-3) -- (-2,-4);
		\draw[thick] (2,-3) -- (4,-4) -- (5,-3);
		\draw[fill=black, radius=.1] (4,-4) circle;
		\tiny\node [above] at (4,-4) {1};
		\draw[thick] (4,-4) -- (4.5,-5);
		\tiny\node [below] at (1,-5) {Tropical projective plane curve $\Gamma$};
		\end{scope}
		\end{tikzpicture}
		\caption{Paired metric graph and pairing morphism.}
		\label{fig:pairedmetricgraph}
	\end{figure}
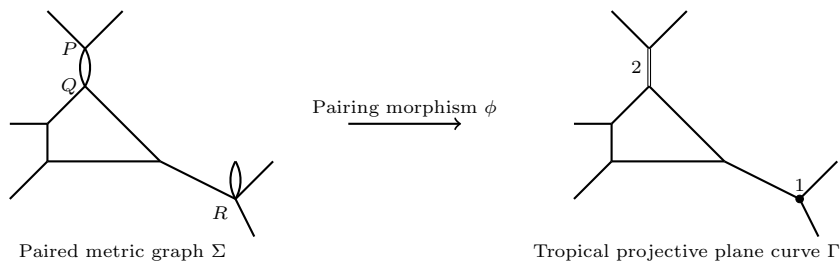

	\begin{definition}
		The metric graph $\Sigma$ constructed this way is called the \emph{paired} metric graph of $\Gamma$, and the map $\phi$ is called the \emph{pairing morphism}. 
	\end{definition}
	
	It is straightforward to check that the genus $g_{\Sigma}$ of $\Sigma$ is exactly $\frac{(d-1)(d-2)}{2}$. 
	%

	\subsection{Divisor theory on metric graphs}
	In this section, we  briefly recall the theory of divisors on metric graphs. A detailed discussion can be found in \cite{BJ, GK}.
	
	Let $\Sigma$ be a metric graph. The $\emph{divisor group }$ $\dv(\Sigma)$ is the free abelian group on the points of $\Sigma$. An element $D$ of $\dv(\Sigma)$ is called a $\emph{divisor}$. Each divisor has a unique representation
	\[
	D=a_1\cdot p_1 + a_2\cdot p_2 + \cdots + a_n\cdot p_n.
	\]
	When the coefficient of $p$ at $D$ is $a$, we write $D(p)=a$.
	The $\emph{degree}$ of $D$ is $\deg(D):=a_1 + a_2 + \cdots a_n$, and $D$ is said to be $\emph{effective}$ if all the coefficients are non-negative.
	
	A $\emph{rational function}$ on $\Sigma$ is a continuous piecewise linear function with integer slopes. For each rational function $f$ on $\Sigma$, we associate a divisor $\ddiv(f)$ where the coefficient at each point $p$ is equal to the sum of incoming slopes of $f$ at $p$. Such divisors are referred to as $\emph{principal}$, and together they  form a subgroup $\prin(\Sigma)$ of $\dv(\Sigma)$. Divisors $D_1$, $D_2$, are said to be \emph{linearly equivalent} if $D_1 - D_2$ is principal. The quotient group $\pic(\Sigma):=\dv(\Sigma)/\prin(\Sigma)$ is called the $\emph{Picard group}$ of $\Sigma$. The restriction to  divisors of degree zero is the \emph{Jacboian} of $\Sigma$, denoted $\Jac(\Sigma)$. 
	
	A divisor $D$ has rank $r$ if $D-E$ is linearly equivalent to an effective divisor, for \emph{every} effective divisor $E$ of degree $r$, and $r$ is the largest integer with this property.  The $\emph{canonical divisor}$ is any divisor linearly equivalent to 
	\[
	K_\Sigma := \sum_{v\in \Sigma}(\val(v)-2)\cdot v.
	\]
	From the tropical Riemann--Roch theorem \cite{GK}, the degree of the canonical divisor is $2g-2$.  
	
	
	\begin{definition}
		A $\emph{theta characterisctic}$ of a metric graph $\Sigma$ is a divisor $\theta$ which satisfies $2\theta \sim K_{\Sigma}$. 
	\end{definition}

	\subsection{Divisor theory on tropical plane curves}
For a smooth tropical or algebraic curve, linear equivalence of divisors is closely related with intersections: the pullback of two different lines gives rise to equivalent divisors. However, this relation breaks when the curve is not smooth. Therefore,  statements of interest to us will be stated and proved on the paired metric graph and shifted to the tropical plane curve. 

	%
	Let $\Gamma$ be a tropical plane curve, and let $\phi:\Sigma\to\Gamma$ be the paired metric graph. As usual, $\dv(\Gamma)$ is the free abelian group generated by the set of points of $\Gamma$.  However, linear equivalence of divisors needs to be defined differently.	The pairing morphism induces a group homomorphism $\phi_*$ from $\dv(\Sigma)$ to $\dv(\Gamma)$, by defining $\phi_*(D)(P) = \sum_{P'\in\phi^{-1}(P)}D(P')$. 

	
	\begin{definition}
		The Picard group of $\Gamma$ with respect to $\Sigma$ is  
		\[
		\pic(\Sigma,\Gamma):= \dv(\Gamma)/{\phi_*(\prin(\Sigma))}.
		\]
		Similarly, the Jacobian of $\Gamma$ with respect to $\Sigma$ is 
		\[
				J(\Sigma,\Gamma):= \dv_0(\Gamma)/{\phi_*(\prin(\Sigma))}.
		\]
	\end{definition}
\noindent	One easily sees that $\pic(\Sigma,\Gamma) \cong \pic(\Sigma)/{<[Q]-[Q']>_{\phi(Q)=\phi(Q')}}$.
%

	The group $\pic(\Sigma,\Gamma)$ can be identified as the usual Picard group of another metric graph. Let $\Gamma^w$ be the non-weighted metric graph obtained from $\Gamma$ by rescaling every edge $e$ by $\frac{1}{w_e}$ and forgetting the weights. Then there is a natural map between $\Gamma$ and $\Gamma^w$ given by rescaling edges accordingly, that is extended by linearity to a map on divisors. 
	\begin{proposition}\label{prop:AJ}
		The Picard group of $\Gamma^w$  is isomorphic to the Picard group $\pic(\Sigma,\Gamma)$.
	\end{proposition}
	
	\begin{proof}
		It is clear that the map from $\Gamma$ to $\Gamma^w$ induces a bijection between $\dv(\Gamma)$  and $\dv(\Gamma^w)$. We need to show that principal divisors on $\Sigma$ give rise to principal divisor on $\Gamma^w$, and that every principal divisor arises this way.
		Let $D$ be a principal divisor on $\Sigma$ such that $D = \ddiv(f)$ for some $f$, and let $D'$ be the corresponding divisor on $\Gamma^w$. Define a function $f'$ on $\Gamma^w$ whose slope at every point is the sum of the slopes at the corresponding points of $\Sigma$. Then $f'$ defines a continuous piecewise linear function with $\ddiv(f') = D'$. 
		
		Conversely, suppose that $E'=\ddiv(g')$ is a principal divisor on $\Gamma^w$. We define a function $g$ on $\Sigma$ as follows. Let $e$ be an edge of $\Gamma^w$ of length $\ell(e)$ where the slope of $g'$ is $s$. Assume that the corresponding edge of $\Gamma$ has weight $k$. Then, by definition, there are $k$ edges   $e_1,\ldots,e_k$ of  $\Sigma$, each of length $a\cdot\ell(e)$ mapping down to $e$. Divide each of them into $k$ segments of equal length. For each edge $e_i$, we define the slope of $g$ to be $-s$, except for the $i$-th segment, in which the slope is $s\cdot k$. Then $g$ is well defined, continuous, and $\ddiv(g)$ is a divisor mapping down to $D'$. 
	\end{proof}
	
	As a result of Proposition \ref{prop:AJ}, together with \cite[Theorem 3.4]{BakerFaber} , the group  $J(\Sigma,\Gamma)$ is a real torus of dimension $g_\Gamma$.  We define the \emph{canonical} divisor class of  $\Gamma$ as $K_{\Gamma} = \phi_*(K_{\Sigma})$. Note that this is not the same as the canonical divisor of the underlying metric graph of $\Gamma$. 
	Accordingly, we  define a \emph{theta characteristic} to be any divisor class $\theta$ satisfying  $2\theta \sim K_\Gamma$. This definition coincides with the $2$-torsion points of the Jacobian $\Jac(\Sigma,\Gamma)$. As a result, $\Gamma$ has $2^{g_\Gamma}$ theta characteristics.
	
%
%

\begin{definition}\label{def:multiplicity}
		The $\emph{multiplicity}$ of a theta characteristic class $\theta$ on a tropical plane curve is the number of effective theta characteristic classes in its preimage ${\phi_*}^{-1}(\theta)$.
	\end{definition}

Note that $\phi_*:J(\Sigma)\to J(\Sigma,\Gamma)$ is a map of groups. It follows that the fiber of each theta characteristic of $J(\Sigma,\Gamma)$ consists of  $2^{g_{\Sigma}-g_{\Gamma}}$ theta characteristics of $\Sigma$. Therefore, 
	
	\begin{lemma}\label{lem:multiplicity}
		There is a $2^{g_{\Sigma}-g_{\Gamma}}:1$ map from the set of theta characteristic classes on $\Sigma$ to the set of theta characteristic classes on $\Gamma$. Furthermore, only one among the $2^{g_{\Gamma}}$ theta characteristic classes on $\Gamma$ has $2^{g_{\Sigma}-g_{\Gamma}}-1$ effective theta characteristic classes in its preimage, and the others have $2^{g_{\Sigma}-g_{\Gamma}}$ effective theta characteristic classes in its preimage. 
	\end{lemma}

\noindent In particular, when $g(\Gamma)=g(\Sigma)$, all the effective theta characteristics of $\Gamma$ have multiplicity $1$. On the other hand, when $g(\Gamma)<g(\Sigma)$, the image of the non-effective theta characteristic is equivalent to an effective divisor. Indeed, it has at least one effective preimage in $\Sigma$, and the pushforward map sends effective divisors to effective divisors.

\subsection*{Theta characteristics and cycles}
	The theta characteristics of a metric graph $\Sigma$ are in bijection with the $2$-torsion points of its Jacobian, and with with the elements of $H_0(\Sigma,\ZZ_2)$. 	In \cite{Zh}, Zharkov illustrates a combinatorial  recipe for  constructing a theta characteristic from every cycle of $\Sigma$. The trivial cycle gives rise to a characteristic  with $-1$ chip at at every vertex, and a chip in the midpoint of each edge. This is the non-effective theta characteristic. For a non-trivial cycle $\sigma$, an effective theta characteristic $\theta_\sigma$ is constructed as follows. On every point of $\sigma$, the value of $\theta_\sigma$ is $\val(p)-2$, where $\val(p)$ is the valency of $p$. For a point $p$ outside of $\sigma$, let $I_p$ be the number of incoming edges at $p$ from the cycle (in other words, it is the number of edges emanating from $p$ along which the distance from $\sigma$ is decreasing). Then $\theta_\sigma(p) = I_p-1$.
	
To find the theta characteristics of a plane tropical curve  $\Gamma$, we use Zhrakov's algorithm on the paired metric graph $\Sigma$. Whenever two cycles of $\Sigma$ get identified in $\Gamma$, the corresponding cycles are in the kernel of the pushforward map $\phi_*:J(\Sigma)\to J(\Sigma,\Gamma)$. Consequently,  the corresponding  theta characteristics map to the same theta characteristic of $\Gamma$.

\begin{example}
Let $\Sigma$ and $\Gamma$ be as in Figure \ref{fig:pairing}. Let $\sigma_1$ be the large cycle in the middle of $\Sigma$, and $\sigma_2$ the union of the large cycle and the cycle above it. Then $\theta_{\sigma_1} = P+R$ and $\theta_{\sigma_2} = Q+R$. 
The images of $\theta_{\sigma_1}$ and $\theta_{\sigma_2}$ are equivalent divisors in $J(\Sigma,\Gamma)$. This is consistent with the fact that the pairing morphism identifies the cycles $\sigma_1$ and $\sigma_2$.
\end{example}

		\section{Bitangent lines to tropical plane curves}
	In this section, we prove that every bitangent to a tropical plane curve gives rise to a theta characteristic. As always $\Gamma $ denotes a tropical plane quartic curve, $\Sigma$ denotes the paired metric graph, and $\phi$ the pairing morphism. 	
	Given a tropical line $\Lambda$, we obtain a divisor
	\[
	\Lambda\cdot\Gamma = \sum_{P\in\Gamma} a_P\cdot P
	\]
	 on $\Gamma$, where $a_P$ is the multiplicity of the stable intersection  of $\Gamma$ and $\Lambda$ at $P$. We remind the reader that to obtain the stable intersection, choose a vector v such that $\Gamma$ and $\Lambda+\epsilon\cdot v$ intersect properly, and take the limit as $\epsilon$ goes to zero.
	 
	A bitangent of $\Gamma$  is, morally speaking, a line that meets the curve with multiplicity at least $2$ at two or more points, or at one point with multiplicity at least $4$. 	However, we need to be a bit more careful  when dealing with non-transverse intersections. 
	\begin{definition}
	A tropical line $\Lambda$ is \emph{bitangent} to $\Gamma$ at $P$ and $Q$ if there is exists a piecewise linear function $f$ that is constant outside of the intersection, and such that $\Gamma\cdot\Lambda+\dv{f}$ is effective and contains $2P+2Q$. 
		\end{definition}		
\noindent	For instance, when a component of $\Lambda\cap\Gamma$ is an edge, there is a tangency at the midpoint of the edge.

	
	\begin{lemma}\label{lem:lines}
		All the line sections of a plane tropical curve $\Gamma$ are equivalent in $\pic(\Sigma,\Gamma)$. Their pullbacks are equivalent in $\pic(\Sigma)$.
	\end{lemma}
	
	\begin{proof}
		Let $\Lambda_1$, $\Lambda_2$ be tropical lines with defining equations $f_1,f_2$. For each $i$, consider the function $\psi_i$ on $\Sigma$ whose slope at a point $p$ equals the slope of $f_i$ at $\phi(p)$. Then $\phi_*\dv(\psi_i)=\Lambda_i\cdot\Gamma$. Moreover, the two line sections $\Lambda_1\cdot\Gamma$ and $\Lambda_2\cdot\Gamma$ are equivalent since their difference is the pushforward of the piecewise linear function $\dv(\psi_2)-\dv(\psi_1)$.

	\end{proof}

	\begin{lemma}\label{lem:loop}
		Suppose that a tropical line $\Lambda$ meets a tropical curve $\Gamma$ at a point $P$ with positive weight. If $P$ is at the vertex of $\Lambda$, then the stable intersection has multiplicity at least $3$. Otherwise, the stable intersection has multiplicity at least $2$. 
		\end{lemma}

		\begin{proof}
The result follows immediately by applying tropical B\'ezout's theorem \cite[Theorem 1.3.2]{MS} locally at the vertex. 
		\end{proof}

	\begin{proposition}\label{prop:canonicalLines}
		A line section of a tropical quartic is equivalent to the canonical divisor.
	\end{proposition}
	
	\begin{proof}
		 By  B\'ezout's theorem, $\Lambda\cdot \Gamma$ has degree 4. Consider a lift $L$ of $\Lambda\cdot\Gamma$ that is supported  away from the loops of $\Sigma$,  such that any two chips on a weighted edge of $\Gamma$ are lifted to the same edge of $\Sigma$. We claim that $L$ is canonical.		 According to Riemann--Roch,  it is suffices to show it has rank 2. 			Let $P$ and $Q$ be points of $\Sigma$. By Lemma \ref{lem:lines}, it suffices to find a tropical line that meets $\phi(P)$ and $\phi(Q)$ with appropriate multiplicities. We distinguish several cases. 
		 
		 \begin{enumerate}
		 \item Both $P$ and $Q$ are in the interior of the same loop. Consider the line $\Lambda'$ whose vertex is at $\phi(P)$.
Then by Lemma \ref{lem:loop}, the pullback of $\Gamma\cdot\Lambda$ has at least $3$ chips at the base of the loop, and is therefore equivalent to an effective divisor with a chip on $P$ and $Q$. 

\item Both $P$ and $Q$ are in the interior of the same edge that is parallel to a standard direction. Consider the line $\Lambda$ whose vertex coincides with a vertex of the edge, and one of its ends overlaps with the edge. Its stable intersection with the edge is at least $2$ along the vertex, and $1$ along the opposite vertex. The pullback of the line is therefore equivalent to a divisor that contains $P$ and $Q$. 

\item Otherwise, consider the line through $\phi(P)$ and $\phi(Q)$. Then its pullback is equivalent to a divisor that contains $P$ and $Q$.

		 \end{enumerate}
 
%
%
	\end{proof}

As an immediate consequence we conclude,
\begin{corollary}
Every effective theta characteristic of a tropical 	plane quartic corresponds to a bitangent line.
\end{corollary}

It is left to show that every effective theta characteristic can be realized as the intersection with a line. We do that in the next section.

	\section{Families of bitangent lines}	
	In this section, we study the behaviour of bitangents of tropical quartics in families, and reinterpret the multiplicity as the number of bitangents converging to it from nearby curves (cf. \cite[Section 3]{CS}). 	
	In order to discuss families of bitangent lines, we set up some notations.	
	The \emph{combinatorial type} of a weighted metric graph is the finite graph obtained by forgetting the edge lengths.
	Let $\Mp_g$ be the moduli space of planar graphs of genus $g$ (see \cite{BJMS}).  	Given a converging sequence of tropical curves, we may assume, perhaps after passing to a subsequence, that their skeletons all  have the same combinatorial type $G$, and that the combinatorial type of the curve in the limit is obtained by contracting edges of $G$. Therefore, they all live in the closure of the same cone of $\Mp_g$. Similarly, if we consider their paired metric graphs, we may assume that they all live in the same cone of the moduli space of tropical curves $M_g^{\trop}$, and have the same combinatorial type $H$.  We therefore have a canonical identification of the theta characteristics of $\Sigma_n$ and $\Sigma$ with cycles of $G$. 
	
	Let $\pic_2(H,w)$ be the tropical universal Picard space classifying divisor classes of degree $2$ on all the metric graphs of type $H$. Then by \cite{Len1}, this space is compact, and the rank of divisors varies upper semicontinuously. In particular, every sequence of theta characteristics has a subsequence that converges to a theta characteristic. 
	
	Let $\theta$ be an effective theta characteristic of $\Gamma$ of characteristic $k$. 
	 Then there are $k$  theta characteristic $\theta^1,\theta^2,\ldots,\theta^k$ of $\Sigma$ corresponding to $\Lambda$. Each $\theta^i$ is the limit of a sequence of theta characteristics $\theta^i_n$ of $\Sigma_n$. Each of them gives rise to a sequence of bitangent lines $\Lambda_n$ of $\Gamma_n$, which converges to a bitangent of $\Gamma$.
Therefore we have shown, 
	
	\begin{theorem}\label{thm:continuity}
	Each theta characteristic of $\Gamma$ gives rise to a bitangent line. Moreover, if $\Gamma_n$ is a sequence of tropical quartics converging to $\Gamma$, then the multiplicity of $\Lambda$ equals the sum of the multiplicities of the bitangents of $\Gamma_n$ converging to it.
	\end{theorem}

\begin{definition}\label{def:bitangentMultiplicity}	Two bitangents are \emph{equivalent} if they correspond to the same theta characteristic. 
The multiplicity of a bitangent is defined to be the multiplicity of the theta characteristic. 	
\end{definition}
\noindent Equivalently, two bitangents are equivalent  precisely when we can continuously move one of them to the other, while continuously moving the tangency points.

	\begin{example}
		Consider the sequence of tropical quartics $\Gamma_n$ depicted in Figure \ref{fig:sequence}. The height of the upper rectangle is $\frac{1}{n}$, and  the sequence converges to the tropical curve  $\Gamma$ on the right side of the figure.
		 For each $\Gamma_n$, we consider  two bitangent lines $\Lambda_n'$ and $\Lambda_n''$. To avoid overcrowding the picture, we only sketched the vertex of each line. Then  the two sequences of bitangents converge to $\Lambda$, whose multiplicity according to Definition \ref{def:bitangentMultiplicity} is $2$.
		
		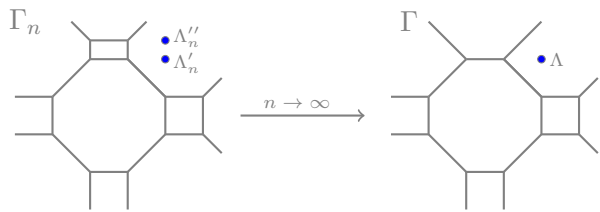
\begin{figure}[h]
			\begin{tikzpicture}[scale=.5,color=gray]

			\begin{scope} 
			
			\node [left] at (-1,1) {$\Gamma_n$};
			\draw[thick] (0,0) -- (0,.5);
			\draw[thick] (0,0) -- (-1,-1);
			\draw[thick] (0,0) -- (1,0);
			\draw[thick] (1,0) -- (1,.5);
			\draw[thick] (1,0) -- (2,-1);
			\draw[thick] (0,.5) --(1,.5);
			\draw[thick] (0,.5)--(-.5,1);
			\draw[thick] (1,.5)--(1.5,1);
			\draw[thick] (-1,-1)--(-1,-2)--(0,-3)--(1,-3)--(2,-2)--(2,-1)--(1,0);
			\draw[thick] (-1,-1)--(-2,-1);
			\draw[thick] (-1,-2)--(-2,-2);
			\draw[thick] (2,-1)--(3,-1);
			
			\draw[thick] (3,-1)--(3,-2);
			\draw[thick] (3,-1)--(3.5,-.5);
			\draw[thick] (3,-2)--(3.5,-2.5);
			
			\draw[thick] (2,-2)--(3,-2);
			\draw[thick] (0,-3)--(0,-4);
			\draw[thick] (1,-3)--(1,-4);
			
			\draw[thick, ->] (4,-1.5) -- (7.3,-1.5);
			\tiny\node [above] at (5.5, -1.5) {$n\to\infty$};
			
			\draw [fill=blue,radius=.1] (2,0) circle;
			\draw [fill=blue,radius=.1] (2,.5) circle;
			\tiny
			\node [right] at (2,-.1) {$\Lambda'_n$};
			\node [right] at (2,.6) {$\Lambda''_n$};

			\end{scope}
			
			\begin{scope}[shift={(10,0)}] 
			\node [left] at (-1,1) {$\Gamma$};
			\draw[thick] (0,0) -- (-1,-1);
			\draw[thick] (0,0) -- (1,0);
			\draw[thick] (1,0) -- (2,-1);
			\draw[thick] (0,0)--(-1,1);
			\draw[thick] (1,0)--(2,1);
			\draw[thick] (-1,-1)--(-1,-2)--(0,-3)--(1,-3)--(2,-2)--(2,-1)--(1,0);
			\draw[thick] (-1,-1)--(-2,-1);
			\draw[thick] (-1,-2)--(-2,-2);
			\draw[thick] (2,-1)--(3,-1);
			\draw[thick] (2,-2)--(3,-2);
			\draw[thick] (0,-3)--(0,-4);
			\draw[thick] (1,-3)--(1,-4);
			
			\draw[thick] (3,-1)--(3,-2);
			\draw[thick] (3,-1)--(3.5,-.5);
			\draw[thick] (3,-2)--(3.5,-2.5);
			
			\draw [fill=blue,radius=.1] (2,0) circle;
			\tiny\node [right] at (2,0) {$\Lambda$};
			
			\end{scope}
			
			\end{tikzpicture}
			\caption{A converging sequence of tropical quartics.}
			\label{fig:sequence}
		\end{figure}
	\end{example}

We are now ready to prove  Theorem \ref{thm:mainTheorem}. 

\begin{theorem}\label{thm:mainTheorem}
		Let $\Gamma$ be a tropical plane quartic curve with genus g. If g is 3, it has 7 equivalence classes of bitangent lines each with multiplicity 1. If g is less than 3, it has $2^g$ equivalence classes of bitangent lines: only one of them has multiplicity $2^{3-g}-1$, and all the others have multiplicity $2^{3-g}$.
	\end{theorem}

\begin{proof}
Let $\Gamma$ be a tropical quartic. By Lemma \ref{lem:multiplicity}, it has one effective theta characteristic of multiplicity $2^{3-g}-1$, and all others have multiplicity $2^{3-g}$. By Theorem \ref{thm:continuity}, each of them corresponds to a bitangent line of the same multiplicity. %
\end{proof}

\begin{remark}
Most of the ingredients of the proof of \cite[Theorem 3.9]{BLMPR} carry well to the non-smooth case. Therefore,  our Theorem \ref{thm:mainTheorem} could also be proven directly, by making the necessary adjustments. 
\end{remark}

We finish by asking about the relation between algebraic and tropical bitangents.  When $\Gamma$ is smooth, there are $4$ bitangents of $C$ tropicalizing to every bitangency class of $\Gamma$ \cite{JL, LM}. 
\begin{conjecture}
		Let $C$ be an algebraic curve over the field of Puiseux series, and let  $\Gamma = \trop(C)$. Let $[\Lambda]$ be an equivalence class of bitangents of $\Gamma$ of total multiplicity $m$. Then there are $4m$ bitangents of $C$ tropicalizing to $[\Lambda]$.
	\end{conjecture}

	\bibliographystyle{alpha}
	\bibliography{bibfile}

\end{document}